%% file: main.tex
\documentclass{article}
\usepackage{algorithm}
\usepackage{algpseudocode}
\usepackage{amsthm}
\usepackage{amsmath, amssymb}
\usepackage{xcolor}
\usepackage{enumitem} 
\usepackage{natbib}
\usepackage[margin=3.2cm]{geometry}
\usepackage{times}
\usepackage[hidelinks]{hyperref}

\usepackage{tikz}
\usepackage{pgfplots}
\pgfplotsset{compat=1.17}

\newtheorem{definition}{Definition}
\newtheorem{lemma}{Lemma}

\newtheorem{corollary}{Corollary}
\newtheorem{theorem}{Theorem}

\newtheorem{example}{Example}

\title{The $e$-Partitioning Principle of False Discovery Rate Control\footnote{This paper has been subsumed into the merged work \citet{xu2025bringing}. Please read and cite that paper instead of this one.}
}
\author{Jelle Goeman\footnote{Leiden University Medical Center, Leiden, The Netherlands} \and Rianne de Heide\footnote{University of Twente, Enschede, The Netherlands} \and Aldo Solari\footnote {Ca' Foscari University of Venice, Venice, Italy}}

\begin{document}
\maketitle

\begin{abstract}
We present a novel necessary and sufficient principle for False Discovery Rate (FDR) control. This $e$-Partitioning Principle says that a procedure controls FDR if and only if it is a special case of a general $e$-Partitioning procedure. By writing existing methods as special cases of this procedure, we can achieve uniform improvements of these methods, and we show this in particular for the eBH, BY and Su methods. We also show that methods developed using the $e$-Partitioning Principle have several valuable properties. They generally control FDR not just for one rejected set, but simultaneously over many, allowing post hoc flexibility for the researcher in the final choice of the rejected hypotheses. Under some conditions, they also allow for post hoc adjustment of the error rate, choosing the FDR level $\alpha$ post hoc, or switching to familywise error control after seeing the data. In addition, $e$-Partitioning allows FDR control methods to exploit logical relationships between hypotheses to gain power.
\end{abstract}

\section{Introduction}

For familywise error rate (FWER) control, it has long been known \citep{Sonnemann1982, Sonnemann2008} that there is a single universal principle that is necessary and sufficient for the construction of valid methods. This Closure Principle says that any method that controls FWER is a special case of a closed testing procedure \citep{marcus1976closed}. The Closure Principle was later challenged by the seemingly more powerful Partitioning Principle \citep{finner2002partitioning}, but \citet{goeman2021only} have shown that the two principles are equivalent. \citet{genovese2006exceedance} and \citet{goeman2011multiple} have extended closed testing to control of False Discovery Proportions (FDPs), and \citet{goeman2021only} showed that all methods controlling a quantile of the distribution of FDP are either equivalent to a closed testing procedure or are dominated by one, extending the Closure Principle to all methods controlling FDP. 

The Closure Principle is useful for method development in FWER and FDP in several ways. In the first place, it reduces the complex task of constructing a multiple testing method to the simpler task of choosing hypothesis tests for intersection hypotheses. After making this choice, method construction reduces to a discrete optimization problem. Moreover, the Closure Principle helps to handle complex situations such as restricted combinations, i.e., logical implications between hypotheses \citep{shaffer1986modified, goeman2021only}. Finally, methods constructed using closed testing often allow for some user flexibility, permitting researchers to modify the results of the multiple testing procedure post-hoc without compromising error control \citep{goeman2011multiple}. 

False Discovery Rate (FDR) methods control the expectation of FDP rather than a quantile \citep{benjamini1995controlling}. Such methods are therefore outside of the scope of the Closure Principle. For the construction of such methods, \citet{blanchard2008two} formulated two quite general sufficient conditions, self-consistency and dependence control, under which, if both hold, FDR control is guaranteed. This Self-Consistency Principle simplifies the proof of well-known FDR-controlling procedures such as BH \citep{benjamini1995controlling} and BY \citep{benjamini2001control} and has been seminal for the creation of many others. An important recent example is the eBH procedure \citep{wang2022false}, that controls FDR on the basis of per-hypothesis $e$-values rather than $p$-values. The concept of the $e$-value, a random variable with expectation at most 1 under the null hypothesis, tunes well with FDR since both concepts are expectation-based. Several other authors \citep{ignatiadis2024values, lee2024boosting, ren2024derandomised, ignatiadis2024compound} have pointed out useful connections between FDR and $e$-values.

However, the Self-Consistency Principle is sufficient but not necessary for FDR control, and there is some indication that it is not optimal. Exchangeable methods designed using the self-consistency principle control FDR at the more stringent level $\pi_0\alpha$ rather than at $\alpha$, but \citet{solari2017minimally} have shown that any method controlling FDR at $\pi_0\alpha$ can be uniformly improved by a method that controls FDR at $\alpha$, and that does not adhere to the Self-Consistency Principle. Some methods constructed using that principle, e.g., eBH and BY, have a reputation for low power \citep{lee2024boosting, xu2023more}. Furthermore, methods created using the Self-Consistency Principle lack some of the properties that methods created using the Closure Principle do get. Full post-hoc user flexibility, for example, as offered by closed testing \citep{goeman2011multiple}, is unknown in combination with FDR control, although some FDR methods allow some user flexibility \citep{lei2018adapt, lei2021general, katsevich2020simultaneous, katsevich2023filtering}. Moreover, so far no FDR-controlling methods have been proposed that can deal with restricted combinations. 

In this paper, we present a novel principle that is both necessary and sufficient for FDR control. We call this principle the $e$-Partitioning Principle of FDR control because of its similarity to the Partitioning Principle formulation of the Closure Principle for FWER and FDP, and because it is based on $e$-values.
Like the Closure Principle, the new $e$-Partitioning Principle facilitates the task of constructing a multiple testing procedure by reducing it to the simpler task of choosing appropriate $e$-values for specifically constructed partitioning hypotheses. Once these $e$-values are chosen, the multiple testing procedure reduces to a discrete optimization problem. Like the Closure Principle for FWER and FDP, the Partitioning Principle for FDR offers some post-hoc user flexibility in a natural way, and handles restricted combinations without effort. It can be used to formulate uniform improvements of existing methods, including the eBH and BY procedures and the procedure of \citet{su2018fdr}.

The outline of the paper is as follows. We first revisit the definition of FDR control, generalizing that concept to allow for simultaneity in the rejections. Next, we formulate the general $e$-Partitioning procedure and establish that its necessity and sufficiency: the actual $e$-Partitioning Principle. Sections \ref{sec: e-combining} and \ref{sec: p-combining} explore existing methods for FDR control and investigate whether they can be improved. Section \ref{sec: Shaffer} shows how the $e$-Partitioning Principle allows restricted combinations to be used in combination with FDR. In Sections \ref{sec: flexible} and \ref{sec: alpha} we establish different ways in which $e$-Partitioning brings flexibility to the researcher in terms of choice of error rates and rejected sets. We end with some numerical illustrations.

\section{Contributions}

The novel contributions of this paper are the following.

\begin{enumerate}
    \item We formulate the general $e$-Partitioning procedure and prove the $e$-Partitioning Principle the says that any method controls FDR if and only if it is a special case of that general procedure (Section \ref{sec: principle}).
    \item We generalize the concept of FDR control to allow methods to control FDR for more than one set simultaneously (Section \ref{sec: FDR}), and show that such simultaneity comes for free (without reducing $\alpha$-levels) as a consequence of the $e$-Partitioning Principle. We explore the uses of this simultaneity in Section \ref{sec: flexible}.
    \item We use the $e$-Partitioning Principle to construct substantial uniform improvements of the eBH, BY and Su methods. Our improvements always allow more flexibility inf the rejected sets, and often allow rejection of larger sets than the original methods at the same nominal FDR level (Sections \ref{sec: e-combining} and \ref{sec: p-combining}). Polynomial time algorithms for these methods are given (Section \ref{sec: computation}).
    \item We show how power can be gained in FDR control methods for the case of restricted combinations (Section \ref{sec: Shaffer}).
    \item We show two ways in which $e$-Partitioning allows researcher flexibility in the choice of error rates. In the first place, researchers may switch to FWER control, if they reject a smal enough set (Section \ref{sec: flexible}). Secondly, we show that for some FDR controlling methods, researchers may choose the $\alpha$-level at which FDR is controlled post hoc (Section \ref{sec: alpha}).
\end{enumerate}

\section{False Discovery Rate control} \label{sec: FDR}

Let us first define the criterion of FDR control. We will introduce a generalized form of the usual definition of FDR control that allows methods to claim FDR control for more than one set simultaneously. The resulting simultaneous control will be mathematically easier to work with. Moreover, it will allow some post hoc user flexibility, as we will explore further in Section \ref{sec: flexible}.

We will use a set-centered notation for the theory we develop. Throughout the paper, we will denote all sets with capital letters (e.g., $R$), collections of sets in calligraphic (e.g., $\mathcal{R}$), and scalars in lower case (e.g., $r$). We use the shorthand $[i]$ for $\{1,2,\ldots, i\}$. The power set of $R$ is $2^R$. Random variables and random sets will be in boldface (e.g., $\mathbf{e}$, $\mathbf{R}$). Direct inequalities between random variables should always be understood to hold surely, i.e., for all events, unless stated otherwise. We denote $\max(a,b)$ as $a \vee b$.

Let the statistical model $M$ be a set of probability measures, e.g., $M = \{\mathrm{P}_\theta\colon \theta \in \Theta\}$ in a parametric model. Hypotheses are restrictions to the model $M$, and therefore subsets of $M$. Suppose that we have $m$ hypotheses $H_1, \ldots, H_m \subseteq M$ of interest. For every $\mathrm{P} \in M$, some hypotheses are true and others false; let $N_\mathrm{P} = \{i\colon \mathrm{P} \in H_i\}$ be the index set of the true hypotheses for $\mathrm{P}$. 

If we choose to reject the hypotheses with indices in $R \subseteq [m]$, we say that we make $|R|$ discoveries, of which $|R \cap N_\mathrm{P}|$ are false.  The false discovery proportion (FDP) for $R$ and $\mathrm{P}$, therefore, is $|R \cap N_\mathrm{P}|/|R|$, if $R \neq \emptyset$, and conventionally defined as 0 if $R = \emptyset$. If $\mathbf{R}$ is random, then FDP is random. A random $\mathbf{R}$ controls the False Discovery Rate (FDR) at $\alpha$ if the expected FDP is bounded by $\alpha$. Formally, FDR control is defined in Definition \ref{def: classical FDR}. Let $\mathrm{E}_\mathrm{P}$ denote the expectation under $\mathrm{P}$.

\begin{definition}[Classical FDR control] \label{def: classical FDR}
$\mathbf{R} \subseteq [m]$ controls FDR at level $\alpha$ if, for every $\mathrm{P} \in M$,
\[
\mathrm{E}_\mathrm{P}\bigg(  \frac{|\mathbf{R} \cap N_\mathrm{P}|}{|\mathbf{R}| \vee 1} \bigg)  \leq \alpha.
\]
\end{definition}

In this paper, we prefer to work with a more general definition of FDR control, which refers to a random collection of sets, rather than to a single random set. In the classical Definition \ref{def: classical FDR}, a researcher rejecting all hypotheses with indices in the set $\mathbf{R}$ can expect FDP to be at most $\alpha$. In the simultaneous Definition \ref{def: FDR}, the researcher has a collection $\boldsymbol{\mathcal{R}}$ of such sets to choose from. 

\begin{definition}[Simultaneous FDR control] \label{def: FDR}
$\boldsymbol{\mathcal{R}} \subseteq 2^{[m]}$ controls FDR at level $\alpha$ if, for every $\mathrm{P} \in M$,
\[
\mathrm{E}_\mathrm{P}\bigg( \max_{R \in \boldsymbol{\mathcal{R}}} \frac{|R \cap N_\mathrm{P}|}{|R| \vee 1} \bigg)  \leq \alpha.
\]
\end{definition}

In the remainder of this paper, when we speak about FDR control at level $\alpha$, this will refer to Definition \ref{def: classical FDR} or \ref{def: FDR} depending on whether we speak about an index set or a collection of such sets. While our focus will be on simultaneous FDR control (Definition \ref{def: FDR}), our results will imply corresponding results for classical FDR control (Definition \ref{def: classical FDR}) through a duality between the two definitions, which is made explicit in Lemma \ref{lem: duality}. The proof of this lemma is immediate from the definitions.

\begin{lemma} \label{lem: duality}
If $\mathbf{R}$ controls FDR at $\alpha$, then so does $\boldsymbol{\mathcal{R}} = \{\mathbf{R}\}$;
if $\boldsymbol{\mathcal{R}}$ controls FDR  at $\alpha$, then so does any $\mathbf{R} \in \boldsymbol{\mathcal{R}}$.
\end{lemma}

We remark that $\alpha$ is taken as fixed in Definitions \ref{def: classical FDR} and \ref{def: FDR}, which implies that it is not allowed to depend on the data. In Section \ref{sec: alpha} we will consider a further generalization of Definition \ref{def: FDR} that allows simultaneous FDR control over all $\alpha \in (0,1]$, and therefore random $\boldsymbol\alpha$.

\section{The $e$-Partitioning Principle of FDR control} \label{sec: principle}

The central result of this paper is the $e$-Partitioning Principle of FDR control. Analogous to the Partitioning Principle of FWER control \citep{finner2002partitioning}, this novel principle gives a general recipe for designing FDR control methods. In this section we show that this novel principle is both necessary and sufficient for FDR control. To distinguish the two Partitioning Principles, we will call the Partitioning Principle for FWER control the $p$-Partitioning Principle.

Both Partitioning Principles partition the model $M$ into (maximally) $2^m$ disjoint parts, defined by the equivalence classes of the collection of true hypotheses $N_\mathrm{P}$. For every $S \in 2^{[m]}$, the partitioning hypothesis corresponding to $S$ is
\[
H_S = \{\mathrm{P} \in M\colon N_\mathrm{P} = S\}.
\]
Since the partitioning hypotheses are defined through an equivalence class on $M$, they are disjoint and cover $M$. As a consequence, exactly one among $H_S$, $S \in 2^{[m]}$, is true and all the others are false. We will generally ignore $H_\emptyset$, since it is impossible to make false discoveries if $H_\emptyset$ is true; define $\mathcal{M} = 2^{[m]} \setminus \{\emptyset\}$ as the collection indexing the partitioning hypotheses of interest.

To adapt the Partitioning Principle for FDR control, we combine partitioning with the concept of the $e$-value, building on earlier work \citep{wang2022false, ignatiadis2024compound} that demonstrates a close relationship between FDR control and $e$-values. In particular, \citet[Theorem~6.5]{ignatiadis2024compound} prove that every FDR controlling procedure can be recovered by eBH applied to compound e-values. However, their result is not constructive. 

We call $\mathbf{e}$ an $e$-value for a hypothesis $H$ if $\mathbf{e} \geq 0$ and $\mathrm{E}_\mathrm{P}(\mathbf{e}) \leq 1$ for all $\mathrm{P} \in H$. Hypothesis testing may be based on $e$-values rather than on $p$-values \citep{shafer2021testing,ramdas2023game,grunwals2024safe}. For a single hypothesis $H$, we may reject $H$ when $\mathbf{e} \geq 1/\alpha$, while controlling Type I error at level $\alpha$, since for $\mathrm{P} \in H$, by Markov's inequality,
\[
\mathrm{P}(\mathbf{e} \geq 1/\alpha) \leq \frac{\mathrm{E}_\mathrm{P}(\mathbf{e})}{1/\alpha} \leq \alpha.
\]

Combining $e$-values and partitioning hypotheses, let $\mathbf{e}_S$, for every $S \in \mathcal{M}$, be an $e$-value for the partitioning hypothesis $H_S$. Together, all $e$-values for all partitioning hypotheses form a \emph{suite} of $e$-values $\mathbf{E} = (\mathbf{e}_S)_{S \in \mathcal{M}}$. Choosing $e$-values for partitioning hypotheses may seem like a complex task since partitioning hypotheses may look unusual. However, Lemma \ref{lem: closure} shows that partitioning $e$-values can simply be constructed as compound $e$-values for intersection hypotheses, allowing the literature on this subject to be exploited \citep{vovk2021values, vovk2024merging, wang2024only}. Often, there is no loss in power doing so, but Section \ref{sec: Shaffer} will discuss situations for which there is a meaningful difference between partitioning and intersection hypotheses.

\begin{lemma} \label{lem: closure}
If $\mathbf{e}$ is an $e$-value for $\tilde H_S = \bigcap_{i\in S} H_i$, then it is an $e$-value for $H_S$.
\end{lemma}

\begin{proof}
We have $\mathrm{E}_\mathrm{P}(\mathbf{e}) \leq 1$ for all $\mathrm{P} \in \tilde H_S = \{\mathrm{P} \in M\colon S \subseteq N_\mathrm{P}\} \supseteq H_S$.
\end{proof}

We are now ready to formulate our main result. Based on any suite $\mathbf{E} = (E_S)_{S \in \mathcal{M}}$ we can construct an FDR-controlling procedure as follows. Define the an \emph{$e$-partitioning procedure} as
\begin{equation} \label{eq: partitioning}
\mathcal{R}_\alpha(\mathbf{E}) = \Big\{ R \in 2^{[m]}\colon \alpha \mathbf{e}_S \geq \frac{|R \cap S|}{|R| \vee 1} \textrm{\ for all $S \in \mathcal{M}$}\Big\}.
\end{equation}
To understand why $\mathcal{R}_\alpha(\mathbf{E})$ would control FDR at $\alpha$, note that the FDP $|R \cap N_\mathrm{P}|/|R|$ depends on $\mathrm{P}$ only through $N_\mathrm{P}$. Equation (\ref{eq: partitioning}) bounds the FDP for all $S=N_\mathrm{P}$ by $\alpha \mathbf{e}_S$, which has expectation at most $\alpha$.

We claim not just that (\ref{eq: partitioning}) controls FDR, but that every FDR-control procedure is of the form (\ref{eq: partitioning}). This is the $e$-Partitioning Principle of FDR control, formulated as Theorem \ref{thm: principle}. Analogous to the Closure Principle for FWER and FDP \citep{marcus1976closed, goeman2021only}, it says that the $e$-partitioning procedure is both necessary and sufficient for FDR control. 

\begin{theorem}[The $e$-Partitioning Principle] \label{thm: principle}
$\boldsymbol{\mathcal{R}}$ controls FDR at level $\alpha$ if and only if $\boldsymbol{\mathcal{R}} \subseteq \mathcal{R}_\alpha(\mathbf{E})$ for a suite of $e$-values $\mathbf{E}$.
\end{theorem}

\begin{proof}
Suppose $\mathbf{E} = (\mathbf{e}_S)_{S \in \mathcal{M}}$ is a suite of $e$-values. Choose any $\mathrm{P} \in M$. If $N_\mathrm{P} \neq \emptyset$, let $S=N_\mathrm{P}$. Then $\mathrm{P} \in H_S$, so that
\[
\mathrm{E}_\mathrm{P}\bigg( \max_{R \in \mathcal{R}_\alpha(\mathbf{E})} \frac{|R \cap N_\mathrm{P}|}{|R| \vee 1} \bigg)  \leq 
\mathrm{E}_\mathrm{P} (\alpha \mathbf{e}_S) \leq \alpha,
\]
and the same holds trivially if $N_\mathrm{P} = \emptyset$, so that $\mathcal{R}_\alpha(\mathbf{E})$, and consequently $\boldsymbol{\mathcal{R}}$, controls FDR at level $\alpha$. This proves the ``if'' part. 

Next, suppose $\boldsymbol{\mathcal{R}}$ controls FDR at level $\alpha$. Choose any $S \in \mathcal{M}$, and $\mathrm{P} \in H_S$. Then
\[
\mathbf{e}_S = \frac1\alpha \max_{R \in \boldsymbol{\mathcal{R}}}  \frac{|R \cap S|}{|R| \vee 1} = \frac1\alpha \max_{R \in \boldsymbol{\mathcal{R}}}  \frac{|R \cap N_\mathrm{P}|}{|R| \vee 1}
\]
has $\mathrm{E}_\mathrm{P}(\mathbf{e}_S) \leq 1$. Therefore, $\mathbf{E} = (\mathbf{e}_S)_{S \in \mathcal{M}}$ is a suite of $e$-values. For every $\mathbf{R} \in \boldsymbol{\mathcal{R}}$ and every $S \in \mathcal{M}$, we have, 
\[
\alpha \mathbf{e}_S = \max_{U \in \boldsymbol{\mathcal{R}}}  \frac{|U \cap S|}{|U| \vee 1} \geq \frac{|\mathbf{R} \cap S|}{|\mathbf{R}| \vee 1},
\]
so $\mathbf{R} \in \mathcal{R}_\alpha(\mathbf{E})$. This proves the ``only if'' part.
\end{proof}

It is worth making explicit that Theorem \ref{thm: principle} is not tied to our novel Definition \ref{def: FDR} of simultaneous FDR, as Corollary \ref{cor: principle} below, which combines Theorem \ref{thm: principle} and Lemma \ref{lem: duality}, makes clear. In fact, rewriting an FDR-controlling $\mathbf{R}$ in terms of a suite of $e$-values is often a way of obtaining simultaneous FDR control for a classical FDR control methods in a less trivial way than was done by Lemma \ref{lem: duality}. We will look into this aspect in more detail in Section \ref{sec: flexible}.

\begin{corollary} \label{cor: principle}
$\mathbf{R}$ controls FDR at level $\alpha$ if and only if $\mathbf{R} \in \mathcal{R}_\alpha(\mathbf{E})$ for a suite of $e$-values $\mathbf{E}$.
\end{corollary}

The $e$-Partitioning Principle reduces the task of constructing an FDR control procedure to the much simpler task of constructing $e$-values for partitioning hypotheses. When constructing these $e$-values, the researcher should take into account any knowledge, or lack of it, on the joint distribution of the data. After choosing $e$-values, the remaining task, implementing (\ref{eq: partitioning}), is purely computational. This computation may have exponential complexity, since $2^m-1$ $e$-values must be taken into account. However, computation can be done in polynomial time in important cases, as we shall see in Section \ref{sec: computation}.

One way to view the general method of the $e$-Partitioning Principle (\ref{eq: partitioning}) is that it divides the model $M$ into parts, designs an FDR-bound for each part, and combines the results to an overall FDR-control procedure. When constructing each of the partial FDR-bounds, the method designer may assume $H_S$, which gives access to powerful oracle-like information in the form of exact knowledge of $S=N_\mathrm{P}$, the set of true null hypotheses. This is important information in FDR control, where knowledge of $\pi_{0, \mathrm{P}} = |N_\mathrm{P}|/m$ is often already extremely valuable \citep{storey2004strong, benjamini2006adaptive, blanchard2009adaptive}.

The $e$-Partitioning Principle is also helpful for improving existing methods. The ``only if'' part of Theorem \ref{thm: principle} asserts that every existing method controlling FDR has an implicit suite of $e$-values that can be used to reconstruct the method as a special case of (\ref{eq: partitioning}). These $e$-values may be inefficient, e.g.\ because they have an expectation strictly smaller than 1; in such cases Theorem \ref{thm: principle} can sometimes be used to propose a superior method based on a suite of stochastically larger $e$-values. We will give several examples of such improvements in Sections \ref{sec: e-combining} and \ref{sec: p-combining}. However, it should be noted that the suite of $e$-values constructed in the proof of Theorem \ref{thm: principle} is rather circular and not very insightful. For finding improvements it is better to reverse engineer the proof of the existing method to distill the $e$-values implicitly or explicitly used there, and to try and improve those.

Potential for improvement also arises when we note that the compound $e$-values in (\ref{eq: partitioning}) are never compared to critical values larger than $1/\alpha$. Without loss of power, we may therefore truncate all compound $e$-values at $1/\alpha$. If that operation results in $e$-values with expectation strictly below 1, there may be room for uniform improvement of the method as described above. In \citet{wang2022false} this is done by \emph{boosting} the e-values using a truncation function. However, in Section \ref{sec: alpha} we will see that in some cases there are good reasons to forgo such a truncation.

Finally, it is worth noting that, in the proof of Theorem \ref{thm: principle}, the control of FDR hinges on the validity of the single $e$-value $e_S$, for $S=N_\mathrm{P}$. This implies that relevant properties of $e_S$ translate directly to properties of the FDR control procedure. For example, if $e_{S}$ is an $e$-value only in some asymptotic sense, the FDR control of the resulting procedure converges to $\alpha$ at the same rate as $\mathrm{E}(e_S)$ converges to 1.

Returning to \citet{ignatiadis2024compound}, we conclude that the 
$e$-Partitioning Principle is an operationalization of their insight, generalized to simultaneous FDR according to Definition \ref{def: FDR}. It is ironic, however, that as we will see in the next section, this principle shows that eBH itself is not admissible.

\section{FDR-control by combining $e$-values} \label{sec: e-combining}

As a first application of the $e$-Partitioning Principle we will look at the important special case that we have $e$-values available for the hypotheses $H_1, \ldots, H_m$, and that the FDR-controlling $\boldsymbol{\mathcal{R}}_\alpha$ should be a function of these $e$-values. Let $\mathbf{e}_1 \geq \ldots \geq \mathbf{e}_m$ be $e$-values for $H_1, \ldots, H_m$, respectively, which we assume ordered without loss of generality. We make no assumptions on the joint distribution of these $e$-values. 

For this situation \citet{wang2022false} proposed the eBH procedure. It is essentially the BH procedure \citep{benjamini1995controlling} applied to $\mathbf{p}_1 =1/\mathbf{e}_1, \ldots, \mathbf{p}_m = 1/\mathbf{e}_m$. However, where BH is only valid for $p$-values whose joint distribution satisfies the PRDS condition \citep{benjamini2001control}, the eBH procedure is valid for any joint distribution of the $e$-values. The eBH procedure at level $\alpha$ rejects the set $\mathbf{R}_\alpha = [\mathbf{r}_\alpha]$, where
\begin{equation} \label{eq: eBH}
\mathbf{r}_\alpha = \max\{1 \leq r \leq m\colon r\mathbf{e}_r \geq m/\alpha\},
\end{equation}
or 0 if the maximum does not exist. The set  $\mathbf{R}_\alpha$ controls FDR at level $\alpha$, as proven by \citet{wang2022false}.

We will now use $e$-Partitioning Principle to propose an alternative method for controlling FDR based on $e$-values under arbitrary dependence. We build on the work of \citet{vovk2021values}, who showed that the only admissible exchangeable method for combining arbitrarily dependent $e$-values into a new $e$-value is to average them, mixing, if desired, with the trivial $e$-value of 1. We will use the unmixed average
\begin{equation} \label{eq: mean e}
\mathbf{e}_S = \frac1{|S|} \sum_{i \in S} \mathbf{e}_i
\end{equation}
as an $e$-value for $H_S$, $S \in \mathcal{M}$. This is an $e$-value for $H_S$ by the results of \cite{vovk2021values} and Lemma \ref{lem: closure}, or directly by Lemma \ref{lem: mean e} below.

\begin{lemma} \label{lem: mean e}
    For all $S \in \mathcal{M}$ the $\mathbf{e}_S $ defined in (\ref{eq: mean e}) is an $e$-value for $H_S$.
\end{lemma}

\begin{proof}
$\mathrm{E}_\mathrm{P}(\mathbf{e}_i) = 1$ for all $\mathrm{P} \in H_i \supseteq H_S$, so $\mathrm{E}_\mathrm{P}(\mathbf{e}_S) = \frac1{|S|} \sum_{i \in S} \mathrm{E}_\mathrm{P}(\mathbf{e}_i) \leq 1$ for all $\mathrm{P} \in H_S$.
\end{proof}

We can now apply (\ref{eq: partitioning}) using $\mathbf{E} = (\mathbf{e}_S)_{S \in \mathcal{M}}$ as the suite of $e$-values, to obtain the eBH+ procedure:
\begin{equation} \label{eq: proc mean e}
\mathcal{R}_\alpha(\mathbf{E}) = \bigg\{ R \in 2^{[m]}\colon  \frac\alpha{|S|} \sum_{i \in S} \mathbf{e}_i \geq \frac{|R \cap S|}{|R| \vee 1} \textrm{\ for all $S \in \mathcal{M}$}\bigg\}.
\end{equation}
This procedure controls FDR under any joint distribution of the $e$-values by Theorem \ref{thm: principle} and Lemma \ref{lem: mean e}. Moreover, the procedure uniformly improves upon eBH as Theorem \ref{thm: eBH} asserts, motivating its name of eBH+. First, we must formally define what we mean by a uniform improvement of a classical FDR-control procedure by a simultaneous FDR-control procedure. Definition \ref{def: uniform} makes this explicit. A simultaneous procedure $\boldsymbol{\mathcal{R}}$ that uniformly improves a classical procedure $\mathbf{S}$ must always allow rejection of the same set of hypotheses, and sometimes allow rejection of a strictly larger set.

\begin{definition} \label{def: uniform}
Let $\boldsymbol{\mathcal{R}}$ and $\mathbf{S}$ both control FDR at level $\alpha$. We say that $\boldsymbol{\mathcal{R}}$ uniformly improves $\mathbf{S}$ if (1.) $\mathbf{S} \in \boldsymbol{\mathcal{R}}$; and (2.) there exists an event $E$, such that, if $E$ happens, $\mathbf{S} \subset \mathbf{R} \in \boldsymbol{\mathcal{R}}$.
\end{definition}

Definition \ref{def: uniform} does not explicitly exclude uninteresting ``uniform improvements'' that reject more than the original procedure only in one or more null events. It will be clear from the context that the improvements proposed in this paper are not of that trivial type.

\begin{theorem} \label{thm: eBH}
If $m>1$, eBH+ uniformly improves eBH.   
\end{theorem}

\begin{proof}
Let $\mathcal{R}_\alpha(\mathbf{E})$ ($=$ eBH+) be defined in (\ref{eq: proc mean e}) and  $\mathbf{R}_\alpha$ ($=$ eBH) be defined just above (\ref{eq: eBH}). Suppose $\mathbf{R}_\alpha \neq \emptyset$ and choose any $S \subseteq \mathcal{M}$, we have
\[
\frac\alpha{|S|} \sum_{i \in S} \mathbf{e}_i \geq \frac\alpha{|S|} \sum_{i \in \mathbf{R}_\alpha \cap S} \mathbf{e}_i \geq \frac{|\mathbf{R}_\alpha \cap S|}{|S|} \alpha\mathbf{e}_{|\mathbf{R}_\alpha|} \geq \frac{m|\mathbf{R}_\alpha \cap S|}{|S||\mathbf{R}_\alpha|} \geq \frac{|\mathbf{R}_\alpha \cap S|}{|\mathbf{R}_\alpha|},
\]
which shows that $\mathbf{R}_\alpha \in \mathcal{R}_\alpha(\mathbf{E})$. The same is trivially true if $\mathbf{R}_\alpha=\emptyset$.

To show actual improvement, let $m>1$ and consider the event that $\mathbf{e}_1=(m-1/2)/\alpha$, $\mathbf{e}_2 = 1/(2\alpha)$, $\mathbf{e}_3 = \ldots = \mathbf{e}_m = 0$. Then $\mathbf{R}_\alpha = \emptyset$, but $\{1\} \in \mathcal{R}_\alpha(\mathbf{E})$, because $\frac\alpha{|S|} \sum_{i \in S} \mathbf{e}_i \geq 1/\alpha$ whenever $1 \in S$.
\end{proof}

The improvement from eBH to eBH+ can be substantial, as can be gauged from the proof of Theorem \ref{thm: eBH}, in which each of the four inequalities leaves a substantial amount of room, and each has a different worst case in which it is an equality. The eBH+ often rejects more hypotheses than eBH, and may reject some hypotheses in the event that eBH does not reject any. An extreme example of this is given in Example \ref{ex: eBH+}, in which eBH+ rejects all hypotheses, but eBH none.

\begin{example} \label{ex: eBH+}
Suppose $m >1$, and $(2m-2i+1)/(m\alpha) \leq \mathbf{e}_i < m/(i\alpha)$ for $i=1,\ldots, m$. Then eBH rejects nothing, but eBH+ rejects $[m]$.
\end{example}

\begin{proof}
It is tedious but straightforward to show that $(2m-2i+1)/m < m/i$ for $i=1,\ldots, m$, so that the example is not void. It follows immediately from (\ref{eq: eBH}) that eBH rejects nothing. Choose $R = [m]$ and any non-empty $S \subseteq [m]$ with $s=|S|$. Then
\[
\sum_{i \in S} \frac{\mathbf{e}_i}s \geq \sum_{i = m-s+1}^m \frac{\mathbf{e}_i}s \geq \sum_{i = m-s+1}^m \frac{2m-2i+1}{ms\alpha} =
\frac{2m-2 \frac{m + (m-s+1)}2 +1}{m\alpha} = \frac{s}{m\alpha} = \frac{|R \cap S|}{|R|\alpha}.
\]
\end{proof}

The eBH+ procedure is qualitatively different from the eBH procedure, and has some properties that eBH does not have, and which procedures constructed according to the Self-Consistency Principle in general do not have. In the first place, rejection of a certain set $\mathbf{R}$ depends not only on $e$-values of hypotheses in $\mathbf{R}$ itself, but also on the $e$-values of other hypotheses. To see an example of this, consider the case $m=2$, and $\mathbf{e}_1=9/(5\alpha)$ and $\mathbf{e}_2=0$. In this case, nothing is rejected. However, if we increase $\mathbf{e}_2$ to $1/(5\alpha)$, we obtain $\mathcal{R}_\alpha(\mathbf{E}) = \{\emptyset, \{1\}\}$. Increasing $\mathbf{e}_2$, apparently, facilitates the rejection of a set of hypotheses that does not include $H_2$. Secondly, $\mathcal{R}_\alpha(\mathbf{E})$ may reject hypotheses for which the corresponding $e$-value is less than $1/\alpha$. To see an example, consider $m=2$, $\mathbf{e}_1=3/(2\alpha)$ and $\mathbf{e}_2=1/(2\alpha)$. It is easily checked that $\{1,2\} \in \mathcal{R}_\alpha(\mathbf{E})$, implying that $H_2$ can be rejected with FDR control at $\alpha$, although $\mathbf{e}_2<1/\alpha$. Translated to $p$-values, these properties of the procedure of (\ref{eq: proc mean e}) are shared by adaptive FDR control procedures that plug in an estimate $\hat\pi_0$ of $\pi_{0,\mathrm{P}} = |N_\mathrm{P}|/m$ into a procedure controlling FDR at $\pi_{0,\mathrm{P}}\alpha$ 
\citep{storey2004strong, benjamini2006adaptive, blanchard2009adaptive}. For such procedures, rejection of a hypothesis
may also depend on $p$-values of other hypotheses through $\hat\pi_0$, and the procedure may reject hypotheses with $p$-values up to $\alpha/\hat\pi_0 > \alpha$. The eBH+ procedure can therefore be seen as an adaptive procedure, even though it was not explicitly constructed using an estimate of $\pi_{0,\mathrm{P}}$. 

Where eBH controls FDR according to Definition \ref{def: classical FDR} and rejects only a single set, eBH+ rejects a collection of such sets, following Definition \ref{def: FDR}. This extension to simultaneous FDR control is not simply the trivial extension of Lemma \ref{lem: duality}. To see this, remark that the proof of Theorem \ref{thm: eBH} does not just apply to the set $\mathbf{R}_\alpha$ rejected by eBH, but to any other self-consistent set, i.e., any set $[\mathbf{s}]$ for which $\mathbf{s}\mathbf{e}_\mathbf{s} \geq m/\alpha$. All such sets are therefore rejected by eBH+. We will explore use of the diversity of the collection $\boldsymbol{\mathcal{R}}$ rejected
by eBH+ in more detail in Section \ref{sec: flexible}.

\citet{wang2022false} already observe that eBH can be improved upon, i.e.\ the rejection set can be enlarged while retaining FDR control, by pre-processing the e-values in some way. \citet{wang2022false} themselves introduce \emph{boosting} by truncating e-values, when the marginal distribution of the null $e$-values is known. \citet{lee2024boosting} boost the e-values by conditioning on a specific sufficient statistic. Both of these improvements require assumptions on the distribution of the e-values. \citet{xu2023more}, in contrast, do not assume anything on the distribution of the $e$-values but boost the power of eBH by introducing exogenous randomness in a clever way: \emph{stochastic rounding}. 
While it is evident that truncating the e-values in a suitable way can also improve the power of methods based on the $e$-Partitioning Principle, such as eBH+, it is not straightforwardly clear whether stochastic rounding would give an improvement. In the set-up of stochastic rounding, some e-values are rounded up and some are rounded down to the grid of the eBH thresholds, in such a way that rounding them down will never lead to fewer rejections, and rounding up might lead to more. However, in our method, the magnitude of e-values from hypotheses that are not rejected do influence the rejected set, as in the example above. Investigating the influence of the various types of boosting and stochastic rounding on methods created by the $e$-Partitioning Principle is a direction for future work.

\section{FDR-control by combining $p$-values} \label{sec: p-combining}

Many FDR control procedures start from $p$-values $\mathbf{p}_1 \leq \ldots \leq \mathbf{p}_m$ for hypotheses $H_1, \ldots, H_m$. We will assume these $p$-values ordered without loss of generality. There are several famous procedures, the choice of which depends on the assumptions that we are willing to make on the joint distribution of the $p$-values. Here, we will explore two such procedures. The first is BY \citep{benjamini2001control}, which is valid for any distribution of the $p$-values. The second is the procedure of \citet{su2018fdr} based on his based on the FDR Linking Theorem. The Su method is valid under the PRDN assumption, a weaker variant of the PRDS assumption onderlying BH. We will place these two methods in the context of the $e$-Partitioning Principle and show that they can be uniformly improved.

Since the $e$-Partitioning Principle depends on $e$-values, we will need to convert the input $p$-values to $e$-values. We use $p$-to-$e$ converters for this \citep{shafer2011test}. A function $e(\mathbf{p})$ is a $p$-to-$e$ calibrator if $e(\mathbf{p})$ is an $e$-value whenever $\mathrm{p}$ is a $p$-value, i.e., whenever $\mathrm{P}(\mathbf{p}\leq t) \leq t$ for all $0 \leq t \leq 1$. One straightforward way to construct an FDR control procedure would be to calibrate $\mathbf{p}_1 \leq \ldots \leq \mathbf{p}_m$ to $e$-values and apply eBH+. However, in general, this does not turn out to be the most efficient approach.

\subsection{BY+: FDR control under general dependence}

The BY method of \citet{benjamini2001control} rejects the set $\mathbf{R}_\alpha = [\mathbf{r}_\alpha]$, where
\begin{equation} \label{eq: BY}
\mathbf{r}_\alpha = \max\{1\leq r\leq m\colon mh_m\mathbf{p}_i \leq r\alpha\},
\end{equation}
or 0 if the maximum does not exist, and $h_m = \sum_{i=1}^m 1/i$ is the $m$th harmonic number. As proven by \citet{benjamini2001control}, $\mathbf{R}_\alpha$ controls FDR for any joint distribution of the $p$-values.

To place the BY procedure into the context of the $e$-Partitioning Principle, we define the following $e$-value, motivated by \citet{xu2024post}, and building on the grid harmonic $p$-to-$e$ calibrator of \citet{vovk2022admissible}. This $e$-value averages, for every $S$, $e$-values obtained by applying a $p$-to-$e$ calibrator to the $p$-values of hypotheses in $S$; however, this $p$-to-$e$ calibrator depends on $S$ as well as on $\alpha$.

\begin{lemma} Under $H_S$
\begin{equation} \label{eq: e for BY}
\mathbf{e}_S = \sum_{i \in S}  \frac{ 1\{ h_{|S|} \mathbf{p}_i \leq \alpha  \}   }{ \alpha  \lceil |S| h_{|S|} \mathbf{p}_i / \alpha  \rceil }
\end{equation}
is an $e$-value.
\end{lemma}

\begin{proof} 
\citet[Proposition 3]{xu2024post} proved that $$e(\mathbf{p}) = \frac{k 1\{ h_k \mathbf{p} \leq \alpha  \}   }{ \alpha  \lceil k h_k  \mathbf{p} / \alpha  \rceil }$$ is a $p$-to-$e$ calibrator for all $\alpha \in (0,1)$ and all $k \in \mathbb{N}$. We take $k=|S|$, apply the $p$-to-$e$ calibrator to all $\mathbf{p}_i$, $i \in S$, and note that the average of the resulting $e$-values is again an $e$-value.
\end{proof}

Next, we can build an FDR control method by plugging in $\mathbf{E} = (\mathbf{e}_S)_{S \in \mathcal{M}}$, for the $e$-values we have just defined in (\ref{eq: e for BY}), into the general $e$-Partitioning Procedure (\ref{eq: partitioning}). We call the resulting procedure BY+ because it uniformly improves BY. This property is proven in the following theorem. 

\begin{theorem}
If $m>1$, BY+ uniformly improves BY.
\end{theorem}

\begin{proof}
Let $\mathcal{R}_\alpha(\mathbf{E})$ ($=$ BY+) be defined in (\ref{eq: e for BY}) and  $\mathbf{R}_\alpha$ ($=$ BY) be defined just above (\ref{eq: BY}). By definition of $\mathbf{r}_\alpha$, we have $\mathbf{p}_i \leq |\mathbf{R}_\alpha|\alpha/mh_m$ for all $i \in \mathbf{R}_\alpha$. Suppose $\mathbf{R}_\alpha\neq\emptyset$ and choose any $S \in \mathcal{M}$. We have
\[
\alpha\mathbf{e}_S 
\geq \sum_{i\in S \cap \mathbf{R}_\alpha} \frac{ 1\{ h_{|S|} \mathbf{p}_i \leq \alpha  \}   }{ \lceil |S| h_{|S|} \mathbf{p}_i / \alpha  \rceil } 
= \sum_{i\in S \cap \mathbf{R}_\alpha} \frac{1} {\lceil |S| h_{|S|} \mathbf{p}_i / \alpha  \rceil }
\geq \sum_{i\in S \cap \mathbf{R}_\alpha} \frac{1} {\lceil |S| h_{|S|} |\mathbf{R}_\alpha| / m h_m  \rceil }
\geq \frac{|S \cap \mathbf{R}_\alpha|} {|\mathbf{R}_\alpha| }.
\]
which shows that $\mathbf{R}_\alpha \in \mathcal{R}_\alpha(\mathbf{E})$. The same is trivially true if $\mathbf{R}_\alpha=\emptyset$.

To show actual improvement, let $m>1$ and consider the event that $\mathbf{p}_1=\alpha/(mh_m)$, $\mathbf{p}_2 = 2\alpha/(mh_m-1)$, $\mathbf{p}_3 = \ldots = \mathbf{p}_m = 1$. Then $\mathbf{R}_\alpha = \{1\}$, but $\{1,2\} \in \mathcal{R}_\alpha(\mathbf{E})$, because $1 \in S$ implies that $\mathbf{e}_S \geq 1/\alpha$, and $2 \in S$ implies that $\mathbf{e}_S \geq 1/(2\alpha)$.
\end{proof}

We now construct an example where BY fails to reject any hypothesis, while BY+ rejects all but one.

\begin{example}
Let $m \geq 4$, and consider the event that
\[
\frac{i\alpha}{m h_m} < \mathbf{p}_i \leq \frac{(i+1)\alpha}{m h_m} \quad \text{for } i = 1, \ldots, m-1, \quad \text{and} \quad \mathbf{p}_m > \frac{\alpha}{h_m}.
\]
Then BY rejects nothing, but BY+ rejects $[m-1]$.
\end{example}

\begin{proof}
It is immediate from the definition that the BY rejection set is empty. Let $R=[m-1]$ and choose any non-empty $S \subseteq [m]$. We have \( \mathbf{1}\{ h_{|S|} \mathbf{p}_i \leq \alpha \} = 1 \) for all \( i \leq m-1 \). For \( S \neq [m] \), it follows that
\[
\left\lceil \frac{|S| h_{|S|} \mathbf{p}_i}{\alpha} \right\rceil \leq \left\lceil \frac{|S| h_{|S|} (i+1)}{m h_m} \right\rceil\leq m-1 \quad \text{for all } i \leq m-1.
\]
Hence, for all \( S \neq [m] \), we obtain
\[
\alpha \mathbf{e}_S \geq \sum_{i \in S \setminus \{m\} } \frac{1}{\left\lceil \frac{|S| h_{|S|} \mathbf{p}_i}{\alpha} \right\rceil} \geq \frac{|S \setminus \{m\}|}{m-1} = \frac{|S \cap R|}{|R|}.
\]
In case \( S = [m] \), we have
\[
\alpha \mathbf{e}_{[m]} = \sum_{i = 1}^{m-1} \frac{1}{\left\lceil \frac{m h_m \mathbf{p}_i}{\alpha} \right\rceil} \geq \sum_{i = 2}^{m} \frac{1}{i} \geq 1.
\]
This establishes that \( [m-1] \in \mathcal{R}_\alpha(\mathbf{E}) \).
\end{proof}

The improvement of BY+ upon BY is similar in spirit to the improvement of eBH+ relative to eBH. 
However, BY+ will never reject hypotheses $H_i$ with $\mathbf{p}_i > \alpha$, since when $i \in R$, for $S=\{i\}$, we have
\[
\alpha \mathbf{e}_{S} = \frac{1\{ p_i \leq \alpha\}}{\lceil \mathbf{p}_i / \alpha \rceil}=0 < \frac{|S \cap R|}{|R|}.
\]
Like eBH+, BY+ rejects a non-trivial collection of sets, following Definition \ref{def: FDR}. The proof of Theorem \ref{thm: eBH} does not just apply to the set $\mathbf{R}_\alpha$ rejected by BY, but to any other self-consistent set, i.e., any set $[\mathbf{s}]$ for which $\mathbf{p}_\mathbf{s} \geq \mathbf{s}\alpha/(mh_m)$. All such sets, and others, are contained in $\mathcal{R}_\alpha(\mathbf{E})$ for BY+ (see also Section \ref{sec: flexible}).


\subsection{Su+: FDR control under the PRDN assumption}

\citet{su2018fdr} investigated FDR control by BH under the PRDN assumption. PRDN is a weaker assumption than the PRDS assumption that is sufficient for BH to be valid \citep{benjamini2001control}. The $p$-values $\mathbf{p}_1\ldots, \mathbf{p}_m$ satisfy PRDS in $M$ if, for every increasing set $A \subseteq \mathbb{R}^m$, for every $\mathrm{P} \in M$, and for every $i \in N_\mathrm{P}$, we have that
\[
\mathrm{P}((\mathbf{p}_1, \ldots, \mathbf{p}_m) \in A \mid \mathbf{p}_i = t) 
\]
is weakly increasing in $t$. PRDS, therefore, is an assumption both on the $p$-values of true and false hypotheses. PRDN, in contrast, is an analogous assumption only on the $p$-values of true hypotheses. The $p$-values $\mathbf{p}_1\ldots, \mathbf{p}_m$ satisfy PRDN in $M$ if, for every $\mathrm{P} \in M$, for every increasing set $A \subseteq \mathbb{R}^{|N_\mathrm{P}|}$, and for every $i \in N_\mathrm{P}$, we have that
\[
\mathrm{P}((\mathbf{p}_j)_{j \in N_\mathrm{P}} \in A \mid \mathbf{p}_i = t) 
\]
is weakly increasing in $t$. As argued by \cite{su2018fdr}, PRDN is a more attractive assumption than PRDS, since generally we do not want to assume anything about $p$-values of false hypotheses.

Where PRDS is the assumption under which BH was proven \citep{benjamini2001control}, PRDN is the sufficient assumption of the Simes test \citep{simes1986improved, su2018fdr}, a test that rejects $H_S$ when 
\[
\mathbf{p}_S = \min_{1\leq i \leq |S|} \frac{|S|\mathrm{p}_{i: S}}i \leq \alpha,
\]
where $\mathbf{p}_{i:S}$ is the $i$th smallest $p$-value among $p_i$, $i \in S$. \citet{su2018fdr} showed that there is a close connection between the BH procedure and the Simes test. 

Moreover, \citet{su2018fdr} proved that when BH is applied at level $q$ under the assumption of PRDN, rather than PRDS, it achieves an FDR of at most $q - q\log(q)$ rather than $q$. Solving 
\[
\alpha = q - q\log(q),
\]
we obtain $q = -\alpha/w(-\alpha/e)$,\footnote{Though we use $e$ both for $p$-to-$e$ calibrators (as a function) and for $\exp(1)$ (as a constant), this should lead to no confusion.} where $w$ is the $-1$ branch of the Lambert W function. It follows that it is possible to control FDR under PRDN by applying the BH procedure at level $\alpha l_\alpha$, where $l_\alpha = -1/w(-\alpha/e)$. We call this the Su procedure. It rejects $\mathbf{R}_\alpha = [\mathbf{r}_\alpha]$, where
\begin{equation} \label{eq: Su}
\mathbf{r}_\alpha = \max\{1\leq r \leq m\colon m\mathbf{p}_r \leq r \alpha l_\alpha\},
\end{equation}
or 0 if the maximum does not exist. For $\alpha = 0.01, 0.05, 0.1$, we have $l_\alpha = 0.131, 0.174, 0.205$, respectively.

In this section we will improve the Su procedure uniformly using the $e$-Partitioning Principle. First, we will define a useful $p$-to-$e$ calibrator, which will allow us to calibrate the Simes $p$-value.

\begin{lemma}
    $e(x) = \min(l_\alpha/x, 1/\alpha)$ is a $p$-to-$e$ calibrator for all $\alpha \in (0,1]$.
\end{lemma}

\begin{proof}
The Lambert W function has the property that $w(x) = x/\exp(w(x))$. Therefore,
\[
\log\Big(\frac{e}{l_\alpha\alpha}\Big) = \log\Big(-\frac{e}{\alpha}w\big(-\frac{\alpha}{e}\big)\Big) = 
w\big(-\frac{\alpha}{e}\big) = \frac{1}{l_\alpha}.
\]
Let $\mathbf{p}$ be a $p$-value and $\mathbf{e} = e(\mathbf{p})$. 
Then 
\[
\mathrm{E}(\mathbf{e}) \leq \int_0^{\alpha l_\alpha} \frac1\alpha\, dp + \int_{\alpha l_\alpha}^1 \frac{l_\alpha}{p}\, dp = l_\alpha - l_\alpha \log(l_\alpha\alpha) = l_\alpha\log\Big(\frac{e}{l_\alpha\alpha}\Big) = 1.
\]
\end{proof}

\begin{corollary}
For all $\mathrm{P} \in H_S$, 
\begin{equation} \label{eq: e Su}
    \mathbf{e}_S =  \min(l_\alpha/\mathbf{p}_S, 1/\alpha)
\end{equation}    
is an $e$-value under PRDN.
\end{corollary}

We can combine the $e$-values from (\ref{eq: e Su}) into a suite $\mathbf{E} = (\mathbf{e}_S)_{S \in \mathcal{M}}$ and use the general $e$-Partitioning procedure (\ref{eq: partitioning}). We call the resulting procedure Su+ because it uniformly improves the Su procedure. This is proven in the following theorem.

\begin{theorem}
If $m > 1$, Su+ uniformly improves Su.
\end{theorem}

\begin{proof}
Let $\mathcal{R}_\alpha(\mathbf{E})$ ($=$ Su+) be defined using (\ref{eq: e Su}) and  $\mathbf{R}_\alpha$ ($=$ Su) be defined just above (\ref{eq: Su}). By definition of $\mathbf{r}_\alpha$, we have $\mathbf{p}_i \leq |\mathbf{R}_\alpha|l_\alpha\alpha/m$ for all $i \in \mathbf{R}_\alpha$. Choose any $S \in \mathcal{M}$ such that $S \cap \mathbf{R}_\alpha \neq \emptyset$. Then
\[
\mathbf{p}_S \leq \mathbf{p}_{|\mathbf{R}_\alpha \cap S|:S}\frac{|S|}{|\mathbf{R}_\alpha\cap S|}
\leq \mathbf{p}_{|\mathbf{R}_\alpha|} \frac{ |S|}{|\mathbf{R}_\alpha\cap S|}
\leq \frac{|\mathbf{R}_\alpha|l_\alpha\alpha}m \frac{|S|}{|\mathbf{R}_\alpha\cap S|}  
\leq \frac{|\mathbf{R}_\alpha|l_\alpha\alpha}{|\mathbf{R}_\alpha\cap S|}.
\]
It follows that $\alpha\mathbf{e}_S \geq |\mathbf{R}_\alpha\cap S|/|\mathbf{R}_\alpha|$ if $S \cap \mathbf{R}_\alpha \neq \emptyset$ and $\mathbf{p}_S \geq l_\alpha\alpha$. The same is trivially true if either $S \cap \mathbf{R}_\alpha = \emptyset$ or $\mathbf{p}_S < l_\alpha\alpha$. Therefore, we have $\mathbf{R}_\alpha \in \mathcal{R}_\alpha(\mathbf{E})$.

To show actual improvement, let $m>1$ and consider the event that $\mathbf{p}_1 = l_\alpha\alpha/m$, $\mathbf{p}_2 = 2l_\alpha\alpha/(m-1)$, $\mathbf{p}_3 = \ldots, \mathbf{p}_m=1$. Then, $\mathbf{R}_\alpha = \{1\}$. However, $1 \in S$ implies that $\mathbf{p}_S = l_\alpha\alpha|S|/m$, so $\mathbf{e}_S = 1/\alpha$, and $2 \in S$ implies $\mathbf{p}_S \leq 2l_\alpha\alpha|S|/(m-1)$, so $\mathbf{e}_S \geq 1/(2\alpha)$. It follows that $\mathcal{R}_\alpha(\mathbf{E}) = \big\{\{1\}, \{1,2\}\big\}$.
\end{proof}

The following example gives an event in which Su rejects only one hypothesis, but Su+ rejects all.

\begin{example}
Let $m > 1$, and consider the event that
\[
p_1 = \frac{l_\alpha \alpha}{m}, \quad \frac{m l_\alpha \alpha}{m - 1} \leq p_2 = \cdots = p_m < \frac{m l_\alpha \alpha}m.
\]
Then Su rejects only one hypothesis, but Su+ rejects $[m]$
\end{example}

\begin{proof}
That Su rejects $[1]$ only is immediate from the definition. Su+ rejects $[m]$, since we have $\mathbf{e}_S = 1/\alpha$ if $1 \in S$, and $\mathbf{e}_S = (m - 1)/(m \alpha)$ if $1 \notin S$. 
\end{proof}


\section{Restricted combinations} \label{sec: Shaffer}

Thus far, we have not made use of the special properties of the partitioning hypotheses. We have created $e$-values for intersection hypotheses $\tilde H_S$ and used then as $e$-values for partitioning hypotheses $H_S$ using Lemma \ref{lem: closure}. However, the use of partitioning hypotheses rather than intersection hypotheses may gain power in certain situations, as is well-known in the context of FWER control \citep{shafer2021testing, finner2002partitioning}. We will now explain how this power gain extends to FDR control through the $e$-Partitioning Principle. 

The partitioning hypothesis $H_S = \{\mathrm{P} \in M\colon N_\mathrm{P}=S\}$ is, in general, smaller than the intersection hypothesis $\tilde H_S=\{\mathrm{P} \in M\colon N_\mathrm{P}\subseteq S\}$. This difference, however, is not always appreciable. For example, suppose $m=2$ and the two hypotheses are $H_1\colon \theta_1=0$ and $H_2\colon \theta_2=0$. Then for $S=\{1\}$ we have $H_S = \{\mathrm{P} \in M\colon \theta_1=0 \textrm{ and } \theta_2\neq 0\}$, while $\tilde H_S = \{\mathrm{P} \in M\colon \theta_1=0\} = H_1$. In most statistical models, it is not possible to construct a more powerful test for $H_{\{1\}}$ than for $H_1$. 

This changes, however, when $H_1$ and $H_2$ are interval hypotheses. If we have $H_1\colon \theta_1\leq 0$ and $H_2\colon \theta_2 \leq 0$, then $H_{\{1\}} = \{\mathrm{P} \in M\colon \theta_1\leq 0 \textrm{ and } \theta_2 > 0\}$ is appreciably smaller than $H_1$, and, as a consequence, we may often formulate a more powerful (stochastically larger) $e$-value for $H_{\{1\}}$ than for $H_1$.

Logical relationships between hypotheses \citep[restricted combinations;][]{shaffer1986modified} can be seen as an extreme case of the same phenomenon, where $H_S$ can even become $\emptyset$ for some $S$. For example, suppose we have $m=3$ and $H_1\colon \theta_1=\theta_2$, $H_2\colon \theta_1=\theta_3$ and $H_3\colon \theta_2=\theta_3$. Then we have $H_{\{1,3\}} = \{\mathrm{P} \in M\colon \theta_1=\theta_2=\theta_3  \textrm{ and } \theta_1 \neq \theta_3\} = \emptyset$. Consequently, $\mathbf{e}_{\{1,3\}} = \infty$ is a valid $e$-value for this hypothesis, and there is no need to take any stochastically smaller $e$-value based on the data.

Taking restricted combinations into account can therefore lead to substantial gains in power. For example, if we test the $m=6$ pairwise comparisons among four parameters $\theta_1, \theta_2, \theta_3, \theta_4$, then out of 63 hypotheses in $\mathcal{M}$, only 13 are non-empty; the other 50 can essentially be ignored by th $e$-Partitioning procedure. Suppose that the hypotheses $H_1: \theta_1=\theta_2$, has an $e$-value of $4/\alpha$, $H_2: \theta_1=\theta_3$ and $H_3: \theta_1=\theta_4$ both have $e$-values of $1/\alpha$, while the other three hypotheses we have $e$-values of 0, then eBH+ would only reject $\mathbf{R} = \{1\}$ when restricted combinations would not be taken into account, but could additionally reject $\mathbf{R} = \{1,2,3\}$ (and all subsets) when they are.

\section{Human-in-the-loop post hoc FDR control} \label{sec: flexible}

Classical FDR control according to Definition \ref{def: classical FDR} provides the researcher of the method with a single random set of hypotheses that they are allowed to reject. Generally, this is the largest set of a pre-specified form that allows FDR control. For example, in BH this set consists of the hypotheses with the $1, 2, \ldots, \mathbf{r}$th smallest $p$-values, for some random $\mathbf{r}$. For Knockoff-based inference they are the set with the $1, 2, \ldots, \mathbf{r}$th knockoff weights. Methods generally return the largest set $\mathbf{R}$ of the chosen form.

Classical FDR control (Definition \ref{def: classical FDR}) guarantees that FDR remains bounded by $\alpha$ only for the set $\mathbf{R}$ returned by the method. However, as argued by \citet{goeman2011multiple}, in applied contexts there may be reasons for a researcher to deviate from that set. For example, the set $\mathbf{R}$ may turn out to be too large, since the researcher may only have a limited budget for follow-up experiments. Alternatively, some of the smallest $p$-values may be suspect, or less interesting, because of secondary characteristics. An important example of this is the volcano plot method popular in bioinformatics, in which researchers take only a subset of the FDR-significant results, discarding the findings with small effect size estimates \citep{ebrahimpoor2021inflated, enjalbert2022powerful}. Such post hoc discarding of results destroys FDR control, however, as the guarantee of Definition \ref{def: classical FDR} is for $\mathbf{R}$ only. \cite{finner2001false} showed how researchers can, intentionally or not, ``cheat with FDR'' by reducing the set $\mathbf{R}$ post hoc. Even reducing the set $\mathbf{R}$ while retaining its pre-specified form, e.g.\ rejecting a set of the $\mathbf{s} < \mathbf{r}$ smallest $p$-values in BH, compromises FDR control if $\mathbf{s}$ is chosen post hoc \citep{katsevich2020simultaneous}.

Deviation from the set $\mathbf{R}$ of the $\mathbf{r}$ smallest $p$-values has been investigated by several authors. However, in most cases such deviation has to be pre-specified, and the researcher must commit to an algorithm for reducing or changing the set $\mathbf{R}$, before seeing the data, i.e., independently of the data \citep{lei2018adapt, katsevich2023filtering}. Some methods allow some interactive post hoc amending of this pre-chosen algorithm based on progressively revealing part of the data \citep{lei2021general, katsevich2020simultaneous}. As far as we are aware, no FDR control methods have so far been proposed that allow researchers to reduce or amend $\mathbf{R}$ after looking at all of the data. Notably, methods controlling tail probabilities of FDP, rather than FDR, all get some post hoc flexibility for free by the Closure Principle \citep{goeman2011multiple, goeman2021only}.

Post hoc flexibility comes as a direct and free consequence of the $e$-Partitioning Principle through its FDR guarantee according to Definition \ref{def: FDR}. Since FDR is controlled over the maximum over all sets in $\mathcal{R}_\alpha(\mathbf{E})$, the researcher is allowed to choose the final rejected set from among the collection $\mathcal{R}_\alpha(\mathbf{E})$ in any desired way, using all information available. Control of FDR according to Definition \ref{def: FDR} is simultaneous over the sets in $\mathcal{R}_\alpha(\mathbf{E})$, and can be used much like simultaneous confidence intervals are. The $e$-Partitioning Principle, therefore, addresses the ``cheating with FDR'' phenomenon in a clear and effective way, since it specifies exactly in what way the researcher may reduce the optimal set $\mathbf{R}$: reductions to $\mathbf{S}$ retain FDR control if and only if $\mathbf{S} \in \mathcal{R}_\alpha(\mathbf{E})$. For the methods we have considered above, eBH+, BY+ and Su+, the collection $\mathcal{R}_\alpha(\mathbf{E})$ is often rich enough to allow plenty of choice for researchers.

Simultaneous control of FDR as offered by Definition \ref{def: FDR} is not just useful for reducing a single FDR-significant set to a single smaller one, but also for the possibility of splitting it into several smaller sets. This can be relevant in fields such as bioinformatics and neuroimaging. In neuroimaging, hypotheses correspond to voxels (3d equivalents of pixels) in the brain, and the significant set $\mathbf{R}$ often splits naturally into several brain areas (``clusters''). In such situations, FDR control on the total set $\mathbf{R}$ is not very informative, since researchers will interpret their results in terms of the clusters \citep{rosenblatt2018all}. Simultaneous FDR control using the $e$-Partitioning Principle allows researchers to claim that these clusters have FDR at most $\alpha$ if and only if these clusters are in $\mathcal{R}_\alpha(\mathbf{E})$. Similar considerations apply in bioinformatics, where hypotheses correspond to molecular markers, which can be meaningfully grouped into sets called pathways \citep{ebrahimpoor2020simultaneous}. FDR inference based on Definition \ref{def: FDR} allows researchers to make claims about such pathways, where classical FDR inference based on Definition \ref{def: classical FDR} does not.

Of special interest are the singleton sets in $\mathcal{R}_\alpha(\mathbf{E})$. Since singleton sets have an FDP of either 0 or 1, there is no distinction between FDR and FWER for such sets. This has the following important implication, made explicit in Theorem \ref{thm: FWER}. The theorem says that when controlling FDR, if the signal is strong enough and the collection $\boldsymbol{\mathcal{R}}$ contains one or more singleton sets, the researcher may choose to reject the union of all these singleton sets and control FWER in addition to FDR. This switch from FDR control to FWER control may be made fully post hoc, after observing all the data and the resulting rejected collection $\boldsymbol{\mathcal{R}}$. 

\begin{theorem} \label{thm: FWER}
Suppose $\boldsymbol{\mathcal{R}}$ controls FDR at level $\alpha$. Define 
\[
\mathbf{R} = \{i \in [m]\colon \{i\} \in \boldsymbol{\mathcal{R}}\}.
\]
Then $\mathbf{R}$ controls FWER at level $\alpha$, that is, for every $\mathrm{P} \in M$,
\[
\mathrm{P}(\mathbf{R} \cap N_\mathrm{P} = \emptyset) \geq 1-\alpha. 
\]
\end{theorem}

\begin{proof}
Let $\boldsymbol{\mathcal{S}} = \{S \in \boldsymbol{\mathcal{R}}\colon |S|=1\}$. Then, for every $\mathrm{P} \in M$,
\[
\mathrm{P}(\mathbf{R} \cap N_\mathrm{P} \neq \emptyset) =
\mathrm{E}_\mathrm{P}\big( \max_{i \in \mathbf{R}} 1\{i \in N_\mathrm{P}\}\big) =
\mathrm{E}_\mathrm{P}\bigg( \max_{R \in \boldsymbol{\mathcal{S}}} \frac{|R \cap N_\mathrm{P}|}{|R| \vee 1} \bigg)  \leq 
\mathrm{E}_\mathrm{P}\bigg( \max_{R \in \boldsymbol{\mathcal{R}}} \frac{|R \cap N_\mathrm{P}|}{|R| \vee 1} \bigg)  \leq \alpha.
\]
\end{proof}

\section{Post hoc choice of $\alpha$} \label{sec: alpha}

Theorem \ref{thm: FWER} gives researchers the exciting option of switching from FDR control to FWER after seeing the data. In this section we shall see that, for certain $e$-Partitioning procedures, there are additional options for adapting the error rate to the data. To achieve this, we will extend the Definition \ref{def: FDR} of FDR control further, building on the work of \cite{grunwald2024beyond} and \cite{koning2023post}.

\begin{definition}[FDR control with post hoc $\alpha$] \label{def: FDR alpha}
For every $\alpha \in (0,1]$, let $\boldsymbol{\mathcal{R}}_\alpha \subseteq 2^{[m]}$. Then $(\boldsymbol{\mathcal{R}}_\alpha)_{\alpha \in (0,1]}$, controls FDR with post hoc $\alpha$ if, for every $\mathrm{P} \in M$,
\[
\mathrm{E}_\mathrm{P}\bigg( \sup_{\alpha \in (0,1]} \max_{R \in \boldsymbol{\mathcal{R}}_\alpha} \frac{|R \cap N_\mathrm{P}|}{\alpha(|R| \vee 1)} \bigg)  \leq 1.
\]
\end{definition}

Where FDR control according to Definition \ref{def: FDR} allows the researcher to choose $\mathbf{R} \in \boldsymbol{\mathcal{R}}_\alpha$ freely for a pre-chosen $\alpha$, FDR control according to Definition \ref{def: FDR alpha} allows the researcher to choose $\boldsymbol{\alpha}$ freely and post hoc, and $\mathbf{R} \in \boldsymbol{\mathcal{R}_\alpha}$ within the sets on offer for that $\boldsymbol{\alpha}$. FDR control with post hoc $\alpha$ implies FDR control in the sense of Definition \ref{def: FDR} with random $\boldsymbol{\alpha}$. Rather than fixing $\alpha$ in advance, Definition \ref{def: FDR alpha}, therefore, allows $\boldsymbol{\alpha}$ to be a random variable that depends on the data in any desired way.

FDR control with post hoc $\alpha$, though seemingly more complicated, follows more or less directly from the $e$-Partitioning Principle. It only adds the additional requirement that the suite $\mathbf{E}$ of $e$-values does not depend on $\alpha$, as the following theorem asserts. The proof is essentially identical to that of Theorem \ref{thm: principle}. Theorem \ref{thm: alpha} extends the corresponding result of \citet[Theorem 9.10]{ramdas2024hypothesis} that $\alpha$ can be chosen post hoc in eBH to general $e$-Partitioning procedures and to simultaneous FDR control. 

\begin{theorem} \label{thm: alpha}
Suppose $\mathbf{E}$ does not depend on $\alpha$. Then $(\mathcal{R}_\alpha(\mathbf{E}))_{\alpha \in (0,1]}$ controls FDR with post hoc $\alpha$.    
\end{theorem}

\begin{proof}
Choose any $\mathrm{P} \in M$, and let $S=N_\mathrm{P}$. Then $\mathrm{P} \in H_S$, so that
\[
\mathrm{E}_\mathrm{P}\bigg( \sup_{\alpha \in (0,1]} \max_{R \in \boldsymbol{\mathcal{R}}_\alpha} \frac{|R \cap N_\mathrm{P}|}{\alpha(|R| \vee 1)} \bigg) \leq 
\mathrm{E}_\mathrm{P} (\mathbf{e}_S) \leq 1.
\]
\end{proof}



\section{Computation} \label{sec: computation}

Checking whether $R \in \mathcal{R}_\alpha(\mathbf{E})$ involves checking exponentially many $e$-values and may therefore take exponential time. However, in special cases, most notably for eBH+, BY+ and Su+ methods proposed above, computation reduces to polynomial time.

It is easy to check that the compound $e$-values of BY+ and Su+ have the property that they are weakly decreasing in the per-hypothesis $p$-values. We can exploit this in the spirit of the shortcuts for closed testing by \citet{dobriban2020fast}. To check for a set $R$ whether $\alpha\mathbf{e}_S \geq |R \cap S|/|R|$, it suffices to check, for each $1\leq a \leq |R|$ and each $0 \leq b \leq m-|R|$, whether the condition holds for the set $S$ consisting of the indices of $a$ largest $p$-values in $R$ and of the $b$ largest $p$-values in $[m]\setminus R$. This takes $O(m^2)$ evaluations of $\mathbf{e}_{S}$. Finding, e.g., the largest set $[r]$ that is rejected by the method then takes $O(m^3)$ such evaluations.

For eBH+ we can do the calculations even faster. We will illustrate this for sets $R$ of the form $[r]$. Define
\[
f_k = \sum_{i=1}^k \mathbf{e}_i.
\]
Then we have $\alpha\mathbf{e}_S \geq |R \cap S|/|R|$, for all $S \in \mathcal{M}$, if and only if, for all $0\leq a< r$ and $r\leq b\leq m$, 
\[
g(a,r,b)  = f_m - f_b + f_r - f_a - \frac{(m-b+r-a)(r-a)}{r\alpha} \geq 0.
\]
It is easily checked that $g(a,r,b)$ is convex in $a$. Therefore, for fixed $r$ and $b$, the minimum of $g$ can be found in $O(\log m) $ time. We can therefore check whether $R \in \mathcal{R}_\alpha(\mathbf{E})$ in $O(m\log m)$ time, and find the largest such $R$ in $O(m^2\log m)$ time by trying all values of $r$ from $m$ downward. In practice, if $r$ is substantially larger than the size of the largest rejected $R$, one quickly finds $a,b$ for which $g(a,r,b)<0$, so that computation time for finding the largest rejected $R$ is often closer to $m\log m$ in practice. Checking whether $R\in \mathcal{R}_\alpha(\mathbf{E})$ for sets $R$ not of the form $[r]$ can also be done in $O(m\log m)$ time following essentially the same reasoning as for sets of the form $[r]$, adapting the definition of $g$ as appropriate.

\section{Numerical illustrations} \label{sec: simulations}

We show the power improvement of eBH+ in comparison with eBH by depicting the size of the largest rejected set of eBH+ in the figures below in a simple z-testing problem with $\alpha=0.05$, taking inspiration from \citet{lee2024boosting}. We take $m=8$ and the set of non-nulls such that $\mu = (A, \ldots, A, 0, \ldots, 0) \in \mathbb{R}^8$, where the number of true null hypotheses is $\pi_0 m$, and where $A$ is the the mean of the non-null hypotheses. The covariance matrix is $\Sigma_{ij} = 0.8^{|i-j|}$ for any $i, j \in [m]$ for the positively dependent data and $\Sigma_{ij: i \neq j} =-0.8/(m-1)$ and ones on the diagonal for the negatively dependent data. We take $\pi_0 = 0.25$ and vary $A \in \{0.125, 0.25, 0.375, 0.5 \}$. After generating $100m$ data points $Z \sim N(\mu, \Sigma)$, we produce e-values by

\[
e_j(Z_j) = \exp(a_jZ_j - a_j^2/2),
\]
where we choose $a_j$ to be equal to the oracle value $A$ for any $j \in [m]$, to compare the best power for both eBH and eBH+. \citet{lee2024boosting} did experiments for different values and found that when $\mu$ is learned via a hold-out set, the performance is quite similar to the optimal value. Since we compare two methods based on the same e-values, there is no harm in choosing $a$ as we like. We average over $1000$ simulations. Additional experiments with $\pi_0 \in \{0.5, 0.75 \}$ and under independence of the e-values showed the same trends. The FDP is (far) below $0.05$ in all simulations.\footnote{R code for reproducing these figures can be found at \url{https://github.com/RianneDeHeide/ePartitioningPrinciple}.}

\begin{figure}
\centering
\input{tikz_power_pos.tex}~\input{tikz_power_neg.tex}
\end{figure}

\section{Discussion}

The $e$-Partitioning Principle resolves the imparity that existed between FWER and FDP tail probability methods on one side and FDR control methods on the other. It is the direct translation of the Closure Principle of these methods to the FDR context. Therefore, it brings many of the boons of the Closure Principle into the realm of FDR. 

The $e$-Partitioning Principle can be used to propose new methods and to possibly improve existing ones. In this paper we have mostly focused on improving existing methods, because was best suitable to showcase the power of the principle. However, we think its greatest strength is actually in developing novel methods. To develop an FDR control method, a researcher only needs to decide how to aggregate the evidence against a partitioning or intersection hypothesis into an $e$-value. After making that choice, the only remaining problem to be solved is a computational one. Aggregating such evidence may often be naturally done in terms of $p$-values. In that case, it may seem an attractive option to convert such $p$-values to $e$-values and apply eBH+, but we have seen for BY+ and Su+ that it is better to aggregate $p$-values directly to compound $e$-values for partitioning hypotheses. When constructing compound $e$-values, assumptions on the joint distribution and on logical relationships between variables may be profitably used.

Aside from possibly improving methods, the $e$-Partitioning Principle brings post hoc flexibility to FDR control on a scale that was previously only known in FDP control. Rather than only a single rejection set, researchers have a choice of many rejection sets to choose from, and they may use all the data to decide which one to report, while still retaining FDR control. If signal is very strong, an attractive option is to switch to FWER control post hoc, or for methods in which the $e$-values do not depend on $\alpha$, to adjust the target FDR level to match the amount of signal in the data.

There are several ways in which this work may be extended. We assumed a finite number of hypotheses. We think that there are no substantial problems to extending the $e$-Partitioning Principle infinite testing problems, but we leave this to future work. This also holds for the online setting. Similarly, it remains to be investigated if and when stochastic rounding is beneficial in combination with the eBH+ procedure and other procedures generated using $e$-Partitioning.

There are many more FDR control methods than we have considered in this paper, but we believe that many of them can be improved using $e$-Partitioning. It will be a challenge to find polynomial time algorithms for some such improvements, but the literature on advanced shortcuts in closed testing may help in this case. For example, we believe that the branch and bound algorithm may be as useful in $e$-Partitioning as it is in closed testing \citep{vesely2023permutation}.

Finally, an important open problem is the question whether and how BH, by far the most popular FDR control method, can be uniformly improved using the $e$-Partitioning Principle, or at least generalized to allow rejection of more than one single set. The fact that BH is known to be inadmissible \citep{solari2017minimally} could be seen as an indication that such an improvement is possible, but so far we have not found a suitable suite of $e$-values that allows this.

\paragraph{Declaration of funding} Rianne de Heide's work was supported by NWO Veni grant number VI.Veni.222.018. 

\bibliographystyle{chicago} 
\bibliography{eBH}    
\end{document}

%% file: tikz_power_pos.tex
\begin{tikzpicture}[scale=0.7]
	\begin{axis}[
	title = Positive dependence,
	xmin = 0.1,
	xmax = 0.525,
	xtick ={0.125, 0.25, 0.375, 0.5},
	xticklabels ={0.125, 0.25, 0.375, 0.5},
	ymin = 0,
	ymax = 1,
	ylabel=Power,
	xlabel=Signal strength $A$,
        height=10cm,
	width=10cm,
		legend style={at={(0.5,-0.15)},
		anchor=north,legend columns=2}
	]
	\addplot[color=blue, very thick,  mark size=2pt, mark=o] coordinates {
(0.125,0.005166667) 
(0.25,0.367833333) 
(0.375,0.806833333) 
(0.5,0.962500000)
};
 \addlegendentry{eBH}
 	\addplot[color=red, very thick, mark size=2pt, mark=o] coordinates {
(0.125,0.01633333) 
(0.25,0.48683333) 
(0.375,0.86583333) 
(0.5,0.97450000)};
 \addlegendentry{eBH+}
	\end{axis}
\end{tikzpicture}

%% file: tikz_power_neg.tex
\begin{tikzpicture}[scale=0.7]
	\begin{axis}[
	title = Negative dependence,
	xmin = 0.1,
	xmax = 0.525,
	xtick ={0.125, 0.25, 0.375, 0.5},
	xticklabels ={0.125, 0.25, 0.375, 0.5},
	ymin = 0,
	ymax = 1,
	ylabel=Power,
	xlabel=Signal strength $A$,
        height=10cm,
	width=10cm,
		legend style={at={(0.5,-0.15)},
		anchor=north,legend columns=2}
	]
	\addplot[color=blue, very thick,  mark size=2pt, mark=o] coordinates {
(0.125,0.0003333333) 
(0.25, 0.3605000000) 
(0.375,0.8256666667) 
(0.5,0.9643333333)
};
 \addlegendentry{eBH} 
 	\addplot[color=red, very thick, mark size=2pt, mark=o] coordinates {
(0.125,0.0003333333) 
(0.25,0.4883333333) 
(0.375,0.8895000000) 
(0.5,0.9766666667)};
 \addlegendentry{eBH+} 
	\end{axis}
\end{tikzpicture}